\documentclass[11pt]{article} 
\usepackage{multicol} 
\usepackage{newlfont}
\usepackage{amsmath,amssymb}
\usepackage{tikz}

\def\rit{\mathbb{R}}
\def\zit{\mathbb{Z}}   
\def\nit{\mathbb{N}}

\def\qit{\mathbb{Q}} 
\def\cit{\mathbb{C}}

\newcommand{\pf}{{\em Proof.~}}
\newcommand{\qed}{\hfill~~\mbox{$\Box$}}

\newenvironment{proof}{\smallskip \noindent \pf}{\qed \bigskip}

\newtheorem{theorem}{Theorem}[section]
\newtheorem{proposition}[theorem]{Proposition}
\newtheorem{definition}[theorem]{Definition}
\newtheorem{lemma}[theorem]{Lemma}
\newtheorem{corollary}[theorem]{Corollary}

\newtheorem{remark}[theorem]{Remark}
\newtheorem{example}[theorem]{Example}

\DeclareMathOperator{\im}{im}
\DeclareMathOperator{\card}{Card}

\DeclareMathOperator{\gr}{gr}
\DeclareMathOperator{\Cone}{Cone}
\DeclareMathOperator{\supp}{supp}
\DeclareMathOperator{\vol}{vol}

\DeclareMathOperator{\Spec}{Spec}

\DeclareMathOperator{\orb}{orb}

\DeclareMathOperator{\conv}{conv}

\begin{document}

\title{\bf A note on the toric Newton spectrum of a polynomial}
\author{\sc Antoine Douai \\
Universit\'e C\^ote d'Azur, CNRS, LJAD, FRANCE \\
Email address: antoine.douai@univ-cotedazur.fr}

\maketitle

\begin{abstract} 
We define the toric Newton spectrum of a polynomial and we give some applications in singularity theory, combinatorics and mirror symmetry. 
\end{abstract}

\section{Introduction}

Let $f$ be a convenient and Newton nondegenerate polynomial in the sense of Kouchnirenko \cite{K}, defined on $\cit^n$ (we recall the basic definitions in Section \ref{sec:NewtonForms}). 
In this text, we define and we study the {\em toric Newton spectrum} of $f$, that is 
the Hilbert-Poincar\'e series $P_{\gr^{\mathcal{N}}(\frac{\mathcal{B}}{\mathcal{L}})}(z)$ of the ring 
$\mathcal{B}/\mathcal{L}$, graded by the Newton filtration $\mathcal{N}$ where 
$\mathcal{B}$ denotes the polynomial ring $\cit [u_1, \cdots , u_n ]$ and $\mathcal{L}$ denotes the ideal generated by the partial derivatives $u_1\frac{\partial f}{\partial u_1 }, \cdots , u_n \frac{\partial f}{\partial u_n}$ of $f$.
Because the ideal $\mathcal{L}$ fits very well with the Newton filtration, this toric Newton spectrum can be easily computed using a Koszul complex provided by \cite{K}, see Theorem \ref{theo:ToricSpectrumSeries}. It turns out that in some cases (mirror symmetry, combinatorics), this toric spectrum is the natural spectrum to consider.

Let us explain more in details our motivations. First, we show that the spectrum at infinity (defined in \cite{Sab} as a global version of the classical singularity spectrum) of a convenient and nondegenerate polynomial can be computed from its toric Newton spectrum: we get in particular in Proposition \ref{prop:SpectrumToricSpectrum} a combinatorial description that matches with a prior result of Y. Matsui and K. Takeuchi \cite[Theorem 5.16]{MSTakeu} (it should be emphasized that our proof is purely algebraic).  Notice that, unlike the spectrum at infinity, the toric Newton spectrum is not symmetric about $n/2$ (in other words, does not satisfy Poincar\'e duality), but this is logical as we basically deal with a non complete situation.

Second, and this is a manifestation of mirror symmetry, we show that this toric Newton spectrum is related with the orbifold cohomology of a stack naturally produced by the Newton polytope $P$ of $f$. More precisely,
let $\Sigma$ be the fan in $\rit^n$ obtained by taking the cones over the proper faces of $P$ not containing the origin 
and let $X_{\Sigma}$ be the toric variety associated with the fan $\Sigma$.
We denote the set of the vertices of $P$ different from the origin by $\mathcal{V}(P)$.
Following \cite{BCS}, we will call the triple $\mathbf{\Sigma} =(\zit^n, \Sigma ,\mathcal{V}(P))$ the {\em stacky fan} of $P$.
We get in Corollary \ref{coro:ToricSpectraOrbifoldCohomology} the formula 
$$P_{\gr^{\mathcal{N}}(\frac{\mathcal{B}}{\mathcal{L}})}(z)=\sum_{\alpha\in\qit} \dim_{\qit} H^{2\alpha }_{\orb} (\mathcal{X}, \qit ) z^{\alpha}$$
where $\mathcal{X}$ is the stack associated with $\mathbf{\Sigma}$ by \cite[Proposition 4.7]{BCS}. One may
expect an isomorphism of $\qit$-graded {\em rings} 
$$H_{\orb}^{2*}(\mathcal{X}, \qit )\stackrel{\cong}{\longrightarrow} \gr^{\mathcal{N}} (\frac{\mathcal{B}}{\mathcal{L}})$$
(a typical example of what could be a mirror isomorphism) and the previous equality gives a clue at the linear level. The correspondence of the products is much more involved: in the complete case (basically if $P$ contains the origin as an interior point and if $f$ is a Laurent polynomial), 
this is achieved in \cite[Theorem 1.1]{BCS} (with a little help from \cite{K}!)
if $f(u)=\sum_{b\in\mathcal{V}(P)} u^b$. See Section \ref{sec:MirrorSymmetry} for a discussion about the subject.
The interest of such an isomorphism is clear, as the product on $\gr^{\mathcal{N}} (\frac{\mathcal{B}}{\mathcal{L}})$ is in principle easy to compute. 

Third, and this is the combinatorial point of view, the toric Newton spectrum counts weighted lattice points in the Newton polytope $P$ of $f$: the toric Newton spectrum determines the $\delta$-vector of $P$, hence its Ehrhart polynomial. This is emphasized in Section \ref{sec:Combinatoire}.

Last, it should be noticed that the previous results have a straightforward {\em local} version, in which case $f$ is a power series with an isolated critical point at the origin: in particular the local toric spectrum has also a combinatorial interpretation. This is discussed in Section \ref{sec:LocalCase}.\\

\noindent {\em Conventions.} Let $n$ be a positive integer. In this text, $\mathcal{B}$ will denote the polynomial ring $\cit [u_1 ,\cdots , u_n ]$ 
and we will write $u^m:= u_1^{m_1}\cdots u_n^{m_n}$ if $m=(m_1 ,\cdots , m_n )\in\nit^n$. \\

\noindent {\bf Acknowledgement.} I thank Prof. K. Takeuchi for pointing to me the fact that the description of the spectrum at infinity obtained from Proposition \ref{prop:ToricSpectrumBox} and Proposition \ref{prop:SpectrumToricSpectrum} can be found (with a very different proof) in \cite{MSTakeu}.

\section{Kouchnirenko's setting}
\label{sec:NewtonForms}

In this section, we set the framework and the notations. We follow \cite{K}.
The {\em support} of $g=\sum_{m\in \nit^n}a_m u^m \in\mathcal{B}$ is
$\supp (g)=\{m\in\nit^n,\ a_m \neq 0\}$ 
and the {\em Newton polytope} of $g$ is the convex hull of $\{0\}\cup \supp (g)$ in $\rit^n_+$. The {\em Newton boundary} of a Newton polytope is the union of its closed faces that do not contain the origin. The polynomial $f$ is {\em convenient} if, for each $i=1, \cdots , n$, 
there exists an integer $n_i \geq 1$
such that the monomial $u_i ^{n_i}$ appears in $f$ with a non-zero coefficient.
The polynomial $f$ is {\em Newton nondegenerate} (for short in this text: nondegenerate) 
if, for each closed face $\Delta$ of the Newton boundary, the system
\begin{equation}\nonumber
(u_1 \frac{\partial f}{\partial u_{1}})_{\Delta}=\cdots =(u_n \frac{\partial f}{\partial u_{n}})_{\Delta}=0
\end{equation}
has no solutions on $(\cit^* )^n$ (we define here $g_{\Delta}=\sum_{m\in \nit^n \cap \Delta} b_m u^m $ if $g=\sum_{m\in \nit^n}b_m u^m \in \mathcal{B}$).

Let $f\in\mathcal{B}$ and let $P$ be its Newton polytope. 
If $F$ is a facet (face of dimension $n-1$) of the Newton boundary of $P$, let $u_F \in \qit^n$ be such that
\begin{equation}\label{eq:PresentationFacet}
F=P\cap \{n\in \rit^n ,\ \langle u_F ,n\rangle = 1 \}.
\end{equation}
The {\em Newton} function $\nu : \nit^n \rightarrow \qit$ of $P$ is defined by $\nu (a):=\max_{F}\langle u_F ,a\rangle$, where the maximum is taken over the facets of the Newton boundary of $P$.
We have $\nu (a+b )\leq \nu (a) +\nu (b)$
with equality if and only if $a$ and $b$ belong to the same cone (by a cone, we mean a cone spanned by the faces of the Newton boundary of $P$). 
For $\alpha\in\qit$, let
\begin{equation}\nonumber
\mathcal{B}_{\alpha}=\{g\in \mathcal{B},\ \supp (g) \in \nu^{-1}(]-\infty ; \alpha ]) \}.
\end{equation} 
We get an increasing filtration $\mathcal{N}_{\bullet}$ of the ring $\mathcal{B}$, indexed by $\qit$, by setting $\mathcal{N}_{\alpha}\mathcal{B}:=\mathcal{B}_{\alpha}$: this is the {\em Newton filtration} of $\mathcal{B}$. 
We put $\mathcal{B}_{<\alpha}=\cup_{\beta <\alpha}\mathcal{B}_{\beta}$, 
$B_{\alpha}=\frac{\mathcal{B}_{\alpha}}{\mathcal{B}_{<\alpha}}$ and $B=\oplus_{\alpha}\mathcal{B}_{\alpha}/\mathcal{B}_{<\alpha}$.
The product in $B$ is described as follows: for $u^m \in \mathcal{B}$, let
$\delta_m =u^m +\mathcal{B}_{<\nu (m)}\in \mathcal{B}_{\nu (m)}/ \mathcal{B}_{<\nu (m)}=B_{\nu (m)}$;
then,
\begin{equation} \label{eq:ProduitB}
\delta_{m_1} . \delta_{m_2} =\left\{ \begin{array}{ll}
\delta_{m_1 +m_2} & \mbox{if $m_1$ and $m_2 $ belong to the same cone,}\\
0 & \mbox{otherwise}.
\end{array}
\right. 
\end{equation}
We will denote by $P_B (z) =\sum_{\alpha\in\qit}\dim B_{\alpha} z^{\alpha}$ the Hilbert-Poincar\'e series of the graded ring $B$. More generally, if $C=\oplus_{\alpha} C_{\alpha}$ is a graded ring, we will put 
$P_C (z) :=\sum_{\alpha\in\qit}\dim C_{\alpha} z^{\alpha}$.

\section{The toric Newton spectrum of a convenient and nondegenerate polynomial}

Let $f$ be  a convenient and nondegenerate polynomial on $\cit^n$. In what follows, $\mathcal{L}$ will denote the ideal generated by the partial derivative 
$u_1 \frac{\partial f}{\partial u_1 }, \cdots , u_n \frac{\partial f}{\partial u_n }$ of $f$.
By projection, the Newton filtration $\mathcal{N}_{\bullet}$ on $\mathcal{B}$ induces
the Newton filtration $\mathcal{N}_{\bullet}$ on $\frac{\mathcal{B}}{\mathcal{L}}$ 
and we get the graded ring
$\gr^{\mathcal{N}}(\frac{\mathcal{B}}{\mathcal{L}})=\oplus_{\alpha}\gr^{\mathcal{N}}_{\alpha}(\frac{\mathcal{B}}{\mathcal{L}})$.

\begin{definition} The toric Newton spectrum of a
convenient and nondegenerate polynomial $f$ is the Hilbert-Poincar\'e series $P_{\gr^{\mathcal{N}}(\frac{\mathcal{B}}{\mathcal{L}})}(z):=\sum_{\alpha\in\qit}\dim \gr^{\mathcal{N}}_{\alpha}(\frac{\mathcal{B}}{\mathcal{L}})z^{\alpha}$.
\end{definition}

\noindent We will also see the toric Newton spectrum as the sequence of rational numbers $\alpha$,
where each $\alpha$ is counted $\dim \gr^{\mathcal{N}}_{\alpha}(\frac{\mathcal{B}}{\mathcal{L}})$-times.

\begin{theorem} \label{theo:ToricSpectrumSeries}
Let $f$ be a convenient and nondegenerate polynomial on $\cit^n$.
Then
\begin{equation}\label{eq:ToricSpectrumFonda}
P_{\gr^{\mathcal{N}}(\frac{\mathcal{B}}{\mathcal{L}})}(z)=(1-z)^n \sum_{v\in \nit^n} z^{\nu (v)}.
\end{equation}
In particular, the toric Newton spectrum of $f$ depends only on its Newton polytope.
\end{theorem}
\begin{proof}
Let $F_i$ be the class of $x_i \frac{\partial f}{\partial x_i }$ in $B_1 =\frac{\mathcal{B}_{1}}{\mathcal{B}_{<1}}$
and let $F$ be the ideal generated by $F_1 ,\cdots , F_n $ in $B$. By \cite[Th\'eor\`eme 2.8]{K}, and because $f$ is nondegenerate, the sequence 
\begin{equation}\label{eq:SuiteExacteFondPol}
0\rightarrow B^{\binom{n}{n}}\rightarrow B^{\binom{n}{n-1}}\rightarrow\cdots\rightarrow B^{\binom{n}{1}}\stackrel{d_1}{\rightarrow} B
\rightarrow B/F \rightarrow 0
\end{equation}
that we get from the Koszul complex of the elements $F_1 , \cdots F_n$ in the ring $B$ is exact.
It follows that the map 
\begin{equation}\label{eq:partial}
\partial : \mathcal{B}^n \rightarrow \mathcal{B}
\end{equation}
defined
by $\partial (b_1 ,\cdots , b_n )=b_1 u_1 \frac{\partial f}{\partial u_1 }+\cdots + b_n u_n \frac{\partial f}{\partial u_n }$
 is {\em strict} with respect 
to the Newton filtration, see \cite[Th\'eor\`eme 4.1]{K}. Therefore
$\gr^{\mathcal{N}}(\frac{\mathcal{B}}{\mathcal{L}})\cong  \frac{B}{F}$
and
\begin{equation}\label{eq:B/F}
P_{\gr^{\mathcal{N}}(\frac{\mathcal{B}}{\mathcal{L}})}(z)=P_{B/F}(z). 
\end{equation}
Now,
by (\ref{eq:SuiteExacteFondPol}), we have also
\begin{equation}\label{eq:PoincareSeriesQuotientPol}
 P_{B/F}(z)=(1-z)^n P_B (z) 
\end{equation}
where $P_B (z)$ denotes the Hilbert-Poincar\'e series of the graded ring $B$.
\end{proof}

For $\Delta$ a closed face of $P$ whose vertices are $b_{i_1} ,\cdots ,b_{i_k}$, let us define
\begin{equation}\label{eq:Box}
\Box (\Delta ):=\{\sum_{\ell =1}^{k} q_{\ell} b_{i_{\ell}},\ q_{\ell} \in [0,1[,\ \ell =1,\cdots , k\}.
\end{equation}
Let $\mathcal{F} (P)$ be the set of the closed faces of $P$ not contained in the union of the hyperplane coordinates.
We have the following combinatorial description of the toric Newton spectrum:

\begin{proposition}\label{prop:ToricSpectrumBox}
Let $f$ be a convenient and nondegenerate polynomial on $\cit^n$ and let $P$ be its Newton polytope. 
Assume that the faces of $\mathcal{F} (P)$ are simplices.
Then
\begin{equation}\nonumber
 P_{\gr^{\mathcal{N}}(\frac{\mathcal{B}}{\mathcal{L}})}(z)=\sum_{q=0}^{n-1} \sum_{\dim \Delta =q} (z-1)^{n-1-q} 
\sum_{v\in\Box (\Delta )\cap \nit^n}z^{\nu (v)}
\end{equation}
where $\Delta\in \mathcal{F} (P)$. 
\end{proposition}
\begin{proof}
For $\Delta$ a closed face of $P$, let $\Cone (\Delta )$ be the union of the half straight lines starting from the origin and passing through $\Delta$ and $B_{\Delta}=\{g\in\mathcal{B}, \ \mbox{supp} (g)\in \nit^n \cap \Cone (\Delta )\}$.
Because the polynomial $f$ is convenient, there is by \cite[Proposition 2.6]{K} an exact sequence of graded $B$-modules
\begin{equation}\nonumber
0\rightarrow B\rightarrow C_{n-1}\rightarrow C_{n-2}\rightarrow\cdots \rightarrow C_0 \rightarrow 0 
\end{equation}
where $C_q =\oplus_{\dim \Delta = q} B_{\Delta}$ for $\Delta\in\mathcal{F}(P)$.
It follows that
$$P_{B} (z)=\sum_{q=0}^{n-1} \sum_{\dim \Delta =q} (-1)^{n-1-q} P_{B_{\Delta}}(z)$$
where $\Delta\in \mathcal{F} (P)$. 
Using (\ref{eq:PoincareSeriesQuotientPol}), we get
\begin{equation}\nonumber
 P_{B/F} (z)=\sum_{q=0}^{n-1} \sum_{\dim \Delta =q} (-1)^{n-1-q} (1-z)^{n-1-q} (1-z)^{q+1}P_{B_{\Delta}}(z).
\end{equation}
Last, 
\begin{equation}\nonumber
(1-z)^{q+1}P_{B_{\Delta}}(z) = (1-z)^{q+1} \sum_{v\in \Cone (\Delta)\cap \nit^n} z^{\nu (v)}=\sum_{v\in \Box (\Delta )\cap \nit^n} z^{\nu (v)}  
\end{equation}
if $\dim \Delta =q$. Now the assertion follows from (\ref{eq:B/F}).
\end{proof}

\begin{remark} \label{rem:MufMuP}
Define $\mu_P :=n! \vol (P)$ where
the volume $\vol (P)$ of the Newton polytope $P$ is normalized such that the volume of the cube is equal to one. Then
$\dim_{\cit}\frac{\mathcal{B}}{\mathcal{L}}= \mu_P$, see \cite{K}, and it follows 
from Proposition \ref{prop:ToricSpectrumBox} that 
$P_{\gr^{\mathcal{N}}(\frac{\mathcal{B}}{\mathcal{L}})}(z)=\sum_{i=1}^{\mu_P} z^{\beta_i}$ where the $\beta_i$'s are nonnegative rational numbers.
\end{remark}

Last, we give some expected consequences of the previous results (in particular we get a complete description of the toric Newton spectrum for $n=2$, compare with \cite[Example 4.17]{DoSa1}).
If $I$ is an interval of $\rit$ and $P_C (z) =\sum_{\alpha\in\qit}\dim C_{\alpha} z^{\alpha}$ we define $P_C^{I}  (z)=\sum_{\alpha\in I}\dim C_{\alpha} z^{\alpha}$.

\begin{corollary}\label{coro:ToricSpectrum01}
Let $f$ be a convenient and nondegenerate polynomial on $\cit^n$ and let $P$ be its Newton polytope.
Then,
\begin{enumerate}
\item $P_{\gr^{\mathcal{N}}(\frac{\mathcal{B}}{\mathcal{L}})} (z) =
P_{\gr^{\mathcal{N}}(\frac{\mathcal{B}}{\mathcal{L}})}^{[0,n[} (z)$ if $P$ is simplicial,
\item $P_{\gr^{\mathcal{N}}(\frac{\mathcal{B}}{\mathcal{L}})}^{[0,1[} (z) =\sum_{v\in N,\ \nu (v)<1} z^{\nu (v)}$,
\item the coefficient of $z$ in $P_{\gr^{\mathcal{N}}(\frac{\mathcal{B}}{\mathcal{L}})}(z)$ is 
equal to the number of lattice points on the Newton boundary of $P$ minus $n$. 
\end{enumerate}
\end{corollary}
\begin{proof} The first assertion follows from Proposition \ref{prop:ToricSpectrumBox} because $\nu (v)<q+1$ if 
$v\in\Box (\Delta )\cap \nit^n$ for a $q$-dimensional face $\Delta$ of $P$.
Because the map (\ref{eq:partial}) is strict with respect to the Newton filtration,
we have $\mathcal{L}(f)\cap \mathcal{N}_{\alpha}\mathcal{B}= \im \partial \cap \mathcal{N}_{\alpha}\mathcal{B} = 
\partial (\mathcal{N}_{\alpha -1}( \mathcal{B}^n))$
and because $\mathcal{N}_{\beta}\mathcal{B}=0$ if $\beta <0$, we get
\begin{equation}\nonumber
\frac{\mathcal{N}_{\alpha}\mathcal{B}}{\mathcal{L}(f)\cap \mathcal{N}_{\alpha}\mathcal{B}+\mathcal{N}_{<\alpha}\mathcal{B}}
= \frac{\mathcal{N}_{\alpha}\mathcal{B}}{\mathcal{N}_{<\alpha}\mathcal{B}}
\end{equation}
for $\alpha < 1$. Thus, $P_{\gr^{\mathcal{N}}(\frac{\mathcal{B}}{\mathcal{L}})}^{[0,1[} (z)=\sum_{\alpha <1}\dim B_{\alpha}z^{\alpha}$ and the second assertion follows. The third one follows from the computation of the coefficient of $z$ in the formula of Theorem \ref{theo:ToricSpectrumSeries}.
\end{proof}

\begin{remark} It follows that the value $0$ appears in the toric Newton spectrum with multiplicity one: in particular,  the toric Newton spectrum is not symmetric about $\frac{n}{2}$. 
\end{remark}

\section{Toric Newton spectrum and spectrum at infinity}

\label{sec:SpectrumPol}

Let $f$ is a convenient and nondegenerate polynomial on $\cit^n$ with Newton polytope $P$ and let
$\mathcal{I}$ be the ideal generated by the partial derivative $\frac{\partial f}{\partial u_1 }, \cdots , \frac{\partial f}{\partial u_n }$ of $f$.
By \cite{K}, $f$ has only isolated critical points and its global Milnor number 
$\mu_f :=\dim_{\cit}\mathcal{B}/(\frac{\partial f}{\partial u_1 }, \frac{\partial f}{\partial u_2 },\cdots , \frac{\partial f}{\partial u_n })$ is finite.
In this text, a {\em geometric spectrum} $\Spec_f$ of $f$ is an ordered sequence of rational numbers
$\alpha_{1}\leq \cdots\leq  \alpha_{\mu_f }$
that we will identify with the generating function
$\Spec_f (z) :=\sum_{i=1}^{\mu_f}z^{\alpha_{i}}$. The specifications are the following:
the $\alpha_i$'s are positive numbers and $\Spec_f (z)=z^n \Spec_f (z^{-1})$.
We can ask moreover that the multiplicity of $\alpha_1$ in $\Spec_f$ is equal to one (normalization):
this normalization is not automatic and appears in the construction of ``canonical'' Frobenius manifolds in the sense of \cite{DoSa1}.

We first recall how to get a geometric spectrum from the Newton filtration.
We define an increasing filtration, indexed by $\qit$, $W_{\bullet}$ on $\mathcal{B} :=\cit [u_1 ,\cdots , u_n ]$ by
\begin{equation}\label{eq:FiltrationTordue}
W_{\alpha}\mathcal{B} :=\{g\in \mathcal{B},\ \supp (gu_1 \cdots u_n ) \in \nu^{-1}(]-\infty ; \alpha ]) \}.
\end{equation} 
By projection, we get a filtration on $\frac{\mathcal{B}}{\mathcal{I}}$.
We set $B^* _{\alpha}:=\gr_{\alpha}^W \frac{\mathcal{B}}{\mathcal{I}}$ and 
$B^* :=\oplus_{\alpha}B^* _{\alpha}$.

\begin{definition}\label{def:NewtonSpectrum}
The Newton spectrum of the convenient and nondegenerate polynomial $f$ is the Hilbert-Poincar\'e series 
$P_{B^*} (z)= \sum_{\alpha\in\qit}\dim B^* _{\alpha} \ z^{\alpha}$.
\end{definition}

\noindent We will also see the Newton spectrum as a sequence of rational numbers with the following property: the frequency of $\alpha$ in the sequence is equal to $\dim B^{*}_{\alpha}$.

In order to justify this definition (in particular to explain the twist by $u_1\cdots u_n$), let us recall the definition of the spectrum at infinity of a convenient and nondegenerate polynomial $f$ defined on $U=\cit^n$ (we use the notations of \cite[2.c]{DoSa1}). 
Let $G$ be the Fourier-Laplace transform of the Gauss-Manin system of $f$ and let $G_{0}$ be its Brieskorn lattice.
Because $f$ is convenient and nondegenerate, $G_0$ is indeed a free $\cit [\theta ]$-module of rank $\mu_f$ and $G=\cit [\theta ,\theta^{-1}]\otimes G_0$: this follows for instance from the strictness of the map 
(\ref{eq:partial}), see \cite[Remark 4.8]{DoSa1}. 
Last, let $V_{\bullet}$ be the $V$-filtration of $G$ at infinity, that is along $\theta^{-1}=0$. It provides,
by projection, a $V$-filtration on 
the $\mu_f$-dimensional vector space $\Omega_{f}:=\Omega^n (U)/df\wedge \Omega^{n-1}(U)=G_0 / \theta G_0$.
The {\em spectrum at infinity}  of $f$ is 
$\Spec_f ^{\infty} (z):=\sum_{i=1}^{\mu_f} z^{\alpha_i}$ where
the sequence 
$\alpha_{1}, \alpha_{2},\cdots , \alpha_{\mu_f}$
of rational numbers is defined by the following property: the frequency of $\alpha$ in the sequence is 
equal to $\dim \gr^{V}_{\alpha}\Omega_{f}$, see \cite[Section 1 and Corollary 10.2]{Sab}
where it is shown that the spectrum at infinity is a geometric spectrum (notice however that it does not satisfy the normalization condition in general, see \cite{DoSa1}).

We define a Newton filtration on $\Omega^n (U)$ by decreeing that the Newton order of the volume form $\frac{du_1}{u_1}\wedge\cdots\wedge \frac{du_n}{u_n}$ is zero:
the Newton filtration $\mathcal{N}_{\bullet}$ on $\Omega^n (U)$ is thus the increasing filtration $\mathcal{N}_{\bullet}$ indexed by $\qit$ such that
\begin{equation}\nonumber
\mathcal{N}_{\alpha}\Omega^n (U):=\{\omega\in\Omega^n (U),\  \supp (\omega ) \in \nu^{-1}(]-\infty ; \alpha ]) \}
\end{equation}
where $\supp (\omega )=\{m+(1,\cdots ,1)\in\nit^n,\ a_m \neq 0\}$
if $\omega =\sum_{m\in \zit^n}a_m u^m du_1\wedge\cdots\wedge du_n$. This filtration induces a filtration on $\Omega_f$ and, by construction, the Newton spectrum of $f$ of Definition \ref{def:NewtonSpectrum} is equal to the spectrum of this Newton filtration on $\Omega_f$. 
With this normalization of the volume form, the Newton filtration satisfies the characteristic properties of the $V$-filtration $V_{\bullet}$ hence is equal to it \cite{Sab}, \cite{DoSa1} and we get (in the local case, that is if $f$ is the germ of a holomorphic function with an isolated critical point at the origin, an analogous result is given in \cite{KV} and \cite{MS}):

\begin{theorem}\cite[Theorem 12.1]{Sab}
\label{theo:SpecVSpecN}
The spectrum at infinity of a convenient and nondegenerate polynomial $f$ is equal to its Newton spectrum: $\Spec_f ^{\infty} (z)=P_{B^*} (z)$.
\end{theorem}

The Newton spectrum (hence the spectrum at infinity) can be computed from the toric Newton spectrum. 
The first step is given by the following counterpart of \cite[Proposition B.1.2.3]{BGMM}:

\begin{lemma}\label{lemma:grWgrN}
Let $f$ be a convenient and nondegenerate polynomial on $\cit^n$.
Then the multiplication by $u_1 \cdots u_n$ induces isomorphisms
\begin{equation}\nonumber
\lambda_{\alpha} :\gr_{\alpha}^W \frac{\mathcal{B}}{\mathcal{I}}\longrightarrow  \gr_{\alpha}^{\mathcal N}\frac{\mathcal{B}u_1 \cdots u_n +\mathcal{L}(f)}{\mathcal{L}(f)}
\subset \gr_{\alpha}^{\mathcal N}\frac{\mathcal{B} +\mathcal{L}(f)}{\mathcal{L}(f)}
\end{equation}
for $\alpha\in\qit$. 
\end{lemma}

\begin{proof} 
Let $g\in\mathcal{B}$ such that $gu_1 \cdots u_n \in \mathcal{L}(f)\cap \mathcal{N}_{\alpha}\mathcal{B}\cap \mathcal{B}u_1 \cdots u_n$. Then, by  the proposition of \cite[Appendice 1]{BGMM}, there exists $g_1, \cdots , g_n \in\mathcal{B}$ such that 
\begin{equation}\label{eq:gu}
gu_1 \cdots u_n =u_1 \frac{\partial f}{\partial u_1} u_2 \cdots u_n g_1 +\cdots +
u_n \frac{\partial f}{\partial u_n} u_1 \cdots u_{n-1} g_n 
\end{equation}
and $\nu (u_2 \cdots u_n g_1 )\leq \alpha -1, \cdots , \nu (u_1 \cdots u_{n-1} g_n )\leq \alpha -1$.
Thus, the multiplication by $u_1 \cdots u_n$ induces 
$\lambda : \frac{\mathcal{B}}{\mathcal{I}}\rightarrow \frac{\mathcal{B}u_1 \cdots u_n}{\mathcal{L}(f)\cap \mathcal{B}u_1 \cdots u_n}$
and injective maps
\begin{equation}\nonumber
\lambda_{\alpha}:\frac{W_{\alpha}\mathcal{B}}{\mathcal{I}\cap W_{\alpha}\mathcal{B}+W_{<\alpha}\mathcal{B}}\rightarrow 
\frac{\mathcal{N}_{\alpha}\mathcal{B}\cap \mathcal{B}u_1 \cdots u_n}{\mathcal{L}(f)\cap\mathcal{N}_{\alpha}\mathcal{B}\cap \mathcal{B}u_1 \cdots u_n +\mathcal N_{<\alpha}\mathcal{B}\cap\mathcal{B}u_1 \cdots u_n}.
\end{equation}
By construction, these maps are onto.
\end{proof}

For $1\leq p\leq n$, let 
$I_p =\{ (i_1 ,\cdots ,i_p ), \ 1\leq i_1 <\cdots <i_p \leq n \}$
and let $P_{\gr^{\mathcal{N}}(\frac{\mathcal{B}}{\mathcal{L}})}^{(i_1 ,\cdots ,i_p )}(z)$ be the toric Newton spectrum of the restriction $f_{(i_1 ,\cdots ,i_p )}$ of $f$ at 
$u_{i_1}=\cdots =u_{i_p}=0$, with the convention $P_{\gr^{\mathcal{N}}(\frac{\mathcal{B}}{\mathcal{L}})}^{(1 ,\cdots , n )}(z)=1$. The restrictions $f_{(i_1 ,\cdots ,i_p )}$ are convenient and nondegenerate if $f$ is so.

\begin{proposition}\label{prop:SpectrumToricSpectrum}
Let $f$ be a convenient and nondegenerate polynomial with Newton polytope $P$. Then 
\begin{equation}\nonumber
 P_{B^*} (z)=
P_{\gr^{\mathcal{N}}(\frac{\mathcal{B}}{\mathcal{L}})}(z)+\sum_{p=1}^{n-1} (-1)^{p}
\sum_{(i_1 ,\cdots ,i_p )\in I_p}  P_{\gr^{\mathcal{N}}(\frac{\mathcal{B}}{\mathcal{L}})}^{(i_1 ,\cdots ,i_p )}(z)+(-1)^n .
\end{equation}
\end{proposition}
\begin{proof}  Assume that 
$gu_1 \cdots u_{n-1} \in\mathcal{L}(f)\cap \mathcal N_{\alpha}\cit [u_1 ,\cdots , u_{n}]u_1 \cdots u_{n-1}$.
Then, by the lemma of \cite[Appendice 1]{BGMM}, there exists $g_1, \cdots , g_n \in\mathcal{B}$ such that 
\begin{equation}\nonumber
gu_1 \cdots u_{n-1} =u_1 \frac{\partial f}{\partial u_1} u_2 \cdots u_{n-1} g_1 +\cdots +
u_{n-1} \frac{\partial f}{\partial u_{n-1}} u_1 \cdots u_{n-2} g_{n-1}+ 
u_n \frac{\partial f}{\partial u_n} u_1 \cdots u_{n-1} g_n .
\end{equation}
Hence
$g =\frac{\partial f}{\partial u_1} g_1 +\cdots +\frac{\partial f}{\partial u_{n-1}} g_{n-1}+
u_n \frac{\partial f}{\partial u_n}  g_n $ and
we deduce the isomorphisms (put $u_n =0$ in the previous formula),
induced by the multiplication by $u_1 \cdots u_{n-1}$ as in Lemma \ref{lemma:grWgrN}, 
\begin{equation}\label{eq:Restriction}
\gr_{\alpha}^W \frac{\cit [\underline{u}']}{(\frac{\partial f}{\partial u_1}( \underline{u}', 0), \cdots , \frac{\partial f}{\partial u_{n-1}}(\underline{u}', 0)) }\cong \gr_{\alpha}^{\mathcal N}\frac{\cit [u_1 ,\cdots , u_{n-1} ]u_1 \cdots u_{n-1}}{\mathcal{L}(f)\cap \cit [u_1 ,\cdots , u_{n-1} ]u_1 \cdots u_{n-1}}
\end{equation}
 where $\underline{u}'=(u_1 ,\cdots , u_{n-1})$.
 We also have 
$P_{B^*}(1)=\mu_f$ because the filtration $W_{\bullet}$ is exhaustive and finite on $\frac{\mathcal{B}}{\mathcal{I}}$. By \cite[Th\'eor\`eme 1.15]{K}, it follows that
\begin{equation}\nonumber
 P_{B^*} (1)= n! V_n -(n-1)! V_{n-1} +\cdots + (-1)^{n-1} V_1 +(-1)^n
\end{equation}
where, for $1\leq q\leq n-1$, $V_q$ is the sum of the $q$-dimensional volumes of the intersection of $P$ with the hyperplane coordinates of dimension $q$.  Because $P_{\gr^{\mathcal{N}}(\frac{\mathcal{B}}{\mathcal{L}})} (1)= n! V_n$,  
a dimension argument, Lemma \ref{lemma:grWgrN} and equation (\ref{eq:Restriction}) provide 
\begin{equation}\nonumber
 P_{\gr^{\mathcal{N}}(\frac{\mathcal{B}}{\mathcal{L}})}(z)=P_{B^*} (z)+\sum_{p=1}^n 
\sum_{(i_1 ,\cdots ,i_p )\in I_p}  P_{B^*}^{(i_1 ,\cdots ,i_p )}(z)
\end{equation}
where $P_{B^*}^{(i_1 ,\cdots ,i_p )}(z)$ is the Newton spectrum of $f_{(i_1 ,\cdots ,i_p )}$, and 
$P_{B^*}^{(1 ,\cdots , n )}(z)=1$. Now we are done by induction because $P_{B^*} (z)=P_{\gr^{\mathcal{N}}(\frac{\mathcal{B}}{\mathcal{L}})}(z)-1$ if $n=1$ by Proposition \ref{prop:ToricSpectrumBox} and Lemma 
\ref{lemma:grWgrN}. 
\end{proof}

\begin{remark} Combining Proposition \ref{prop:ToricSpectrumBox} and Proposition \ref{prop:SpectrumToricSpectrum}
we get \cite[Theorem 5.16]{MSTakeu}.
\end{remark}

\begin{example} \label{ex:ToricNewton}
In the following examples, the toric Newton spectrum is computed using Proposition \ref{prop:ToricSpectrumBox}
and the Newton spectrum is computed using Proposition \ref{prop:SpectrumToricSpectrum}.

\begin{enumerate}

\item Let $f(u,v)=u^2 +u^2 v^2 +v^2$. Then $\mu_P =8$ and $\mu_f =5$. We get 
$$P_{\gr^{\mathcal{N}}(\frac{\mathcal{B}}{\mathcal{L}})} (z)=1+3z+3z^{1/2} +z^{3/2}$$
and
$$P_{B^*} (z)=P_{\gr^{\mathcal{N}}(\frac{\mathcal{B}}{\mathcal{L}})} (z)- (1+z^{1/2})-(1+z^{1/2}) +1=z^{1/2} +3z+z^{3/2}.$$

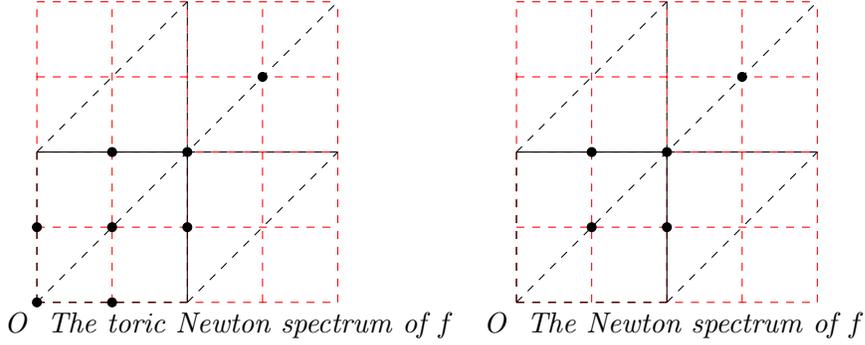
\begin{center}
\begin{tikzpicture}
\draw[dashed,color=red] (0,0) grid (4,4);    
   \draw[domain=0:2] plot (\x,{2}) node[above right]{};
\draw[domain=0:2] plot ({2}, \x) node[above right]{};
\draw[domain=2:4][dashed] plot (\x,{2}) node[above right]{};
\draw[domain=2:4][dashed] plot ({2}, \x) node[above right]{};
 \draw[domain=0:4][dashed] plot (\x,\x) node[below right]{};
 \draw[domain=0:2][dashed] plot (\x,{2+\x}) node[below right]{}; 
\draw[domain=2:4] [dashed] plot (\x , {-2+\x}) node[below right]{}; 
\draw[domain=0:2][dashed] plot (\x,0) node[above right]{};
\draw[domain=0:2][dashed] plot (0,\x) node[above right]{};
\draw[fill=black] (0,0) circle (0.15em);
\draw[fill=black] (0,1) circle (0.15em);
\draw[fill=black] (1,0) circle (0.15em);
\draw[fill=black] (1,1) circle (0.15em);
\draw[fill=black] (2,1) circle (0.15em);
\draw[fill=black] (2,2) circle (0.15em);
\draw[fill=black] (1,2) circle (0.15em);
\draw[fill=black] (3,3) circle (0.15em);
\draw (0,0) node[anchor=north east] {O};
\node[below right] at (0,0) {The toric Newton spectrum of $f$};
   \end{tikzpicture}
\begin{tikzpicture}
   \draw[dashed,color=red] (0,0) grid (4,4);    
   \draw[domain=0:2] plot (\x,{2}) node[above right]{};
\draw[domain=0:2] plot ({2}, \x) node[above right]{};
\draw[domain=2:4][dashed] plot (\x,{2}) node[above right]{};
\draw[domain=2:4][dashed] plot ({2}, \x) node[above right]{};
 \draw[domain=0:4][dashed] plot (\x,\x) node[below right]{};
 \draw[domain=0:2][dashed] plot (\x,{2+\x}) node[below right]{}; 
\draw[domain=2:4] [dashed] plot (\x , {-2+\x}) node[below right]{}; 
\draw[domain=0:2][dashed] plot (\x,0) node[above right]{};
\draw[domain=0:2][dashed] plot (0,\x) node[above right]{};
\draw[fill=black] (1,1) circle (0.15em);
\draw[fill=black] (2,1) circle (0.15em);
\draw[fill=black] (2,2) circle (0.15em);
\draw[fill=black] (1,2) circle (0.15em);
\draw[fill=black] (3,3) circle (0.15em);
\draw (0,0) node[anchor=north east] {O};
\node[below right] at (0,0) {The Newton spectrum of $f$};
   \end{tikzpicture}
\end{center}

\item  Let $f(u,v, w)=u+v+w+u^2 v^2 w^2 +v^2 w^2$. Then $\mu_P =12$ 
and $\mu_f =8$. We first analyze the contributions of 
$v\in\Box (\Delta )$, $\Delta\in\mathcal{F}(P)$:
\begin{itemize}
\item  the contribution of $v=(0,0,0)$ is 
$4+4(z-1)+(z-1)^2 =z^2 +2z +1$,
\item the contribution of $v=(1,1,1)$ is  
$(4+4(z-1)+(z-1)^2 )z^{1/2} =(z^2 +2z +1) z^{1/2}$,
\item the contribution of $v=(1,2,2)$ is
$(2+(z-1))z =(z+1)z =z+z^2$,
\item the contribution of $v=(0,1,1)$ is 
$(2+(z-1))z^{1/2} =(z+1)z^{1/2} =z^{1/2}+z^{3/2}$.
\end{itemize}
\noindent Thus,
$$P_{\gr^{\mathcal{N}}(\frac{\mathcal{B}}{\mathcal{L}})} (z)=1+2z^{1/2}+3z^{3/2}+3z+2z^2+z^{5/2}$$
and
$$P_{B^*}(z) =P_{\gr^{\mathcal{N}}(\frac{\mathcal{B}}{\mathcal{L}})} (z)-(3+z^{1/2}+z+z^{3/2})+3-1=z^{1/2}+2z^{3/2} +2z+2z^2 + z^{5/2}.$$

\end{enumerate}

\end{example}

\section{Applications}

In this section, $f$ is a convenient and nondegenerate polynomial defined on $\cit^n$ and $P$ denotes its Newton polytope. 
Let $\Sigma $ be the fan in $\rit^n$ obtained by taking the cones over the faces of the Newton boundary of $P$ and let $X_{\Sigma}$ be the toric variety associated with the fan 
$\Sigma$. 
We will call $\Sigma $ the {\em fan of $P$} and $X_{\Sigma}$ the {\em toric variety of $P$}.
We will denote by $\Sigma^*$ the set of the cones of $\Sigma$ {\em not contained in the hyperplane coordinates}.

\subsection{Toric Newton spectrum and orbifold cohomology}

\label{sec:NewtonAndOrbifold}

We will use the following version of Hodge-Deligne polynomials.
 Let $v\in \nit^n$ and let $\sigma (v)$ be the smallest cone of $\Sigma$ containing $v$. We will call
\begin{equation}\nonumber
 E_{v} (z):=\sum_{\tau \in \Sigma , \sigma (v)\subseteq \tau } (z-1)^{n-\dim\tau}
\end{equation}
the {\em Hodge-Deligne polynomial of $v$} and
\begin{equation}\nonumber
 E_{v}^* (z):=\sum_{\tau \in \Sigma^* , \sigma (v)\subseteq \tau } (z-1)^{n-\dim\tau}
\end{equation}
the {\em relative Hodge-Deligne polynomial of $v$}.
The case $v=0$ deserves a special attention: we will call $E_{0} (z)$ the Hodge-Deligne polynomial of $X_{\Sigma}$ and $E_{0}^* (z)$ the relative Hodge-Deligne polynomial of $X_{\Sigma}$.

\begin{proposition} \label{prop:DualityHodgeDeligne} 
Let $f$ be a convenient and nondegenerate polynomial  on $\cit^n$ and let $P$ be its Newton polytope.  
Assume that $\Sigma$ is simplicial. Then:
\begin{enumerate}
\item $X_{\Sigma}$ has no odd cohomolgy and
$E_{0}(z)=\sum_{k=0}^n \dim H^{2k}_c (X_{\Sigma}, \qit)z^k$,
\item $z^n E_{0} (z^{-1})= E_{0}^* (z)$,
\item $E_{0}^* (z)=\sum_{i=0}^{n} \dim H^{2i} (X_{\Sigma},\qit )z^i$.
\end{enumerate}
\end{proposition}
\begin{proof}
For the first point we follow closely \cite[Lemma 4.1]{Stapledon}: let us denote by $\rho$ the ray in $-|\Sigma |$ 
whose primitive generator is $v_{\rho}= (-1, \cdots ,-1)\in\zit^n$
and let $\Sigma '$ ({\em resp.} $\Sigma'_{\rho}$) be the fan given by the cones of $\Sigma$ and the cones generated by $\rho$ and the cones of $\Sigma$ contained in the hyperplane coordinates ({\em resp.} be the quotient fan $\Sigma '/ \rho$). 
By the proof of {\em loc. cit.}, we have an exact sequence of cohomology with compact support
\begin{equation}\nonumber
0\longrightarrow H^{2i}_c (X_{\Sigma}, \qit )\longrightarrow H^{2i}_c (X_{\Sigma'}, \qit )\longrightarrow H^{2i}_c (X_{\Sigma'_{\rho}}, \qit )\longrightarrow 0.
\end{equation}
Because $X_{\Sigma'}$ and $X_{\Sigma'_{\rho}}$ are complete and simplicial, the first point holds for $X_{\Sigma'}$ and $X_{\Sigma'_{\rho}}$
(see for instance \cite[section 14]{Da}) and we get  
\begin{equation}\label{eq:Intermediaire}
\sum_{k=0}^n \dim H^{2k}_c (X_{\Sigma}, \qit)z^k =E_{0}^{\Sigma'}(z)-E_{0}^{\Sigma'_{\rho}}(z)
\end{equation}
\begin{equation}\nonumber
=\sum_{\tau\in\Sigma '}(z-1)^{n-\dim \tau}
-\sum_{\tau\in\Sigma', \rho\subseteq \tau}(z-1)^{n-\dim \tau}
=\sum_{\tau\in\Sigma}(z-1)^{n-\dim \tau}=E_{0}(z)
\end{equation}
where $E_{0}^{\Sigma'}(z)$ (resp. $E_{0}^{\Sigma'_{\rho}}(z)$) denotes the Hodge-Deligne polynomial of 
$X_{\Sigma'}$ (resp.  $X_{\Sigma'_{\rho}}$). 
In order to get the second equality, we have used two facts: first, the orbit closure $V(\rho ) =\cup_{\rho\subseteq \tau} O (\tau )$ of the orbit $O(\rho)$ and the toric variety $X_{\Sigma /\rho}$ are isomorphic, see for instance \cite[Proposition 3.2.7]{CLS}; second, the Hodge-Deligne polynomial of $\cit^*$ is $z-1$ and 
$O(\tau )\cong (\cit^* )^{n-\dim \tau }$, see \cite[Section 3.2]{CLS}.
This shows the first point. 

By Poincar\'e duality for the complete and simplicial toric varieties $X_{\Sigma'}$ and $X_{\Sigma'_{\rho}}$ we have 
$z^n E_{0}^{\Sigma'}(z^{-1})=E_{0}^{\Sigma'}(z)$ and $z^{n-1} E_{0}^{\Sigma'_{\rho}}(z^{-1})=E_{0}^{\Sigma'_{\rho}}(z)$.  
Using (\ref{eq:Intermediaire}), we get 
$z^n E_{0} (z^{-1})-E_{0} (z)=(1-z) E_{0}^{\Sigma'_{\rho}}(z)$
thus $z^n E_{0} (z^{-1})=E_{0} (z)-(z-1) E_{0}^{\Sigma'_{\rho}}(z)=E_{0}^* (z)$.

Last, the third point follows from Proposition \ref{prop:DualityHodgeDeligne} and Poincar\'e duality for the simplicial toric variety $X_{\Sigma}$.
\end{proof}

We now express the toric Newton spectrum of $f$ in terms of Hilbert-Poincar\'e series of toric varieties.
Let $b_1 ,\cdots ,b_r \in \mathcal{V}(P)$ be the vertices of $P$ different from the origin.
Let $\Box (P):=\cup_{\Delta\in\mathcal{F}(P)}\Box (\Delta )$
where $\Box (\Delta )$ is defined by (\ref{eq:Box})
and $\mathcal{F} (P)$ denotes the set of the closed faces of $P$ not contained in the union of the hyperplane coordinates.

\begin{theorem}\label{theo:SpectrumToricVariety}
Assume that $\Sigma$ is simplicial. Then
$P_{\gr^{\mathcal{N}}(\frac{\mathcal{B}}{\mathcal{L}})} (z)=\sum_{v\in \Box (P)\cap \nit^n}E_{v}^* (z) z^{\nu (v)}$.
\end{theorem}
\begin{proof} Keeping in mind the definition of $E_{v}^* (z)$, this is precisely Proposition \ref{prop:ToricSpectrumBox}. 
\end{proof}

 \begin{remark} \label{rem:IntegralShifts}
Theorem \ref{theo:SpectrumToricVariety} gives a combinatorial explaination of integral shifts in the toric Newton spectrum: if $v\in \Box (P)\cap N$ and $v$ doesn't belong to the hyperplane coordinates, the sequence 
$\nu (v) , \nu (v)+1 ,\cdots , \nu (v)+n-\dim \sigma (v)$ is a part of the toric Newton spectrum. 
\end{remark}

Let $\mathcal{V} (P)$ be the set of the vertices of $P$ different from the origin.
If $X_{\Sigma}$ is simplicial, one associates to the triple $\mathbf{\Sigma}:=(\zit^n, \Sigma , \mathcal{V} (P))$ a Deligne-Mumford stack $\mathcal{X}$ whose 
coarse moduli space is $X_{\Sigma}$, see \cite{BCS}. We now give a geometric description of the toric Newton spectrum in terms of the orbifold cohomology of $\mathcal{X}$, as defined in \cite[Definition 4.8]{ALR}.

\begin{corollary}\label{coro:ToricSpectraOrbifoldCohomology} 
Let $f$ be a convenient and nondegenerate polynomial on $\cit^n$. 
 Assume that $X_{\Sigma}$ is simplicial. Then 
$P_{\gr^{\mathcal{N}}(\frac{\mathcal{B}}{\mathcal{L}})}(z)=\sum_{\alpha\in\qit} \dim_{\qit} H^{2\alpha }_{\orb} (\mathcal{X}, \qit ) z^{\alpha}$.
\end{corollary}
\begin{proof}
For $v\in\nit^n$, 
let $X_{\Sigma /\sigma (v)}$ be the toric variety associated with the quotient fan $\Sigma/\sigma (v)$ (recall that
$\sigma (v)$ be the smallest cone of $\Sigma$ containing $v$).
By Theorem \ref{theo:SpectrumToricVariety}
and the proof of Proposition \ref{prop:DualityHodgeDeligne}, 
we have 
$$P_{\gr^{\mathcal{N}}(\frac{\mathcal{B}}{\mathcal{L}})}(z)=
\sum_{v\in \Box (P)\cap N} \sum_k \dim_{\qit} H^{2k}(X_{\Sigma/ \sigma (v)}, \qit)z^{k+\nu (v)}$$
$$=
\sum_{\alpha\in\qit}\sum_{v\in \Box (P)\cap N} \dim_{\qit} H^{2(\alpha -\nu (v))}(X_{\Sigma/ \sigma (v)}, \qit)z^{\alpha}$$
by Theorem \ref{theo:SpectrumToricVariety} and Proposition \ref{prop:DualityHodgeDeligne}. 
And $H^{2\alpha}_{\orb} (\mathcal{X}, \qit )=\oplus_{v\in \Box (P)\cap N}H^{2(\alpha -\nu (v))}(X_{\Sigma/ \sigma (v)}, \qit)$ by \cite[Proposition 4.7]{BCS}.
\end{proof}

\begin{remark}\label{rem:OrbiStap}
By Theorem \ref{theo:ToricSpectrumSeries}, the Hilbert-Poincar\'e series $P_{\gr^{\mathcal{N}}(\frac{\mathcal{B}}{\mathcal{L}})}(z)$ is equal to the weighted $\delta$-vector 
denoted by $\delta^0 (t)$ in \cite{Stapledon} and Corollary \ref{coro:ToricSpectraOrbifoldCohomology} could also be deduced from \cite[Theorem 4.3]{Stapledon}. However, we thought that is was useful to give a proof based on Kouchnirenko's work.
\end{remark}

\subsection{Toric Newton spectrum and mirror symmetry}

\label{sec:MirrorSymmetry}

Let $f$ be a convenient and nondegenerate polynomial defined on $\cit^n$ with Newton polytope $P$
and let $\mathcal{X}$ be the Deligne-Mumford stack associated with the stacky fan $\mathbf{\Sigma} =(\zit^n, \Sigma ,\mathcal{V} (P))$ of $P$ as in Section \ref{sec:NewtonAndOrbifold}. As suggested by Corollary \ref{coro:ToricSpectraOrbifoldCohomology}, we 
may expect an isomorphism of $\qit$-graded {\em rings}
\begin{equation}\label{eq:MirrorIso}
 H_{\orb}^{2*}(\mathcal{X}, \qit )\longrightarrow \gr^{\mathcal{N}} (\frac{\mathcal{B}}{\mathcal{L}})
\end{equation}
where as above $\mathcal{L}$ is the ideal generated by the partial derivative 
$u_1 \frac{\partial f}{\partial u_1 }, \cdots , u_n \frac{\partial f}{\partial u_n }$ of $f$.
The product on $\gr (\frac{\mathcal{B}}{\mathcal{L}})$ is in principle easy to compute and should give, with the help of such an isomorphism, a concrete description of
the orbifold product on $H_{\orb}^{2*}(\mathcal{X}, \qit )$. On the other hand, Poincar\'e duality on $H_{\orb}^{2*}(\mathcal{X}, \qit )$ (see for instance \cite[Proposition 4.11]{ALR}) should be helpful in order to understand the pairing on $\gr (\frac{\mathcal{B}}{\mathcal{L}})$ and possible symmetries of the toric Newton spectrum. 
 Because the map (\ref{eq:partial}) is strict with respect to the Newton filtration, the graded ring $\gr (\frac{\mathcal{B}}{\mathcal{L}})$ looks like the "Stanley-Reisner presentation" of $\mathcal{X}$ given by the right hand side of \cite[Theorem 1.1]{BCS}: nevertheless, we can't use directly this result because $X_{\Sigma}$ is not complete.
Notice that if 
$f(u)=\sum_{b\in \mathcal{V}(P)} u^b$ is a convenient and nondegenerate {\em Laurent} polynomial then $X_{\Sigma}$ is complete and the map (\ref{eq:MirrorIso}) is an isomorphism of graded rings by {\em loc. cit.}

\begin{example}

Let us return to the example $f(u,v)=u^2 +u^2 v^2 +v^2$. A basis of $\gr^{\mathcal{N}}\frac{\mathcal{B}}{\mathcal{L}(f)}$ is given by $1, uv , u^2 v^2, u^3 v^3, u, v ,u^2 v ,uv^2 $, with respective grading $0, \frac{1}{2}, 1, \frac{3}{2}, \frac{1}{2} ,\frac{1}{2}, 1 ,1$,
and we get the following table for the product $\bullet$ in $\gr^{\mathcal{N}}\frac{\mathcal{B}}{\mathcal{L}(f)}$:

\begin{center}

\begin{tabular}{|c|c|c|c|c|c|c|c|c|} \hline\hline
$ \bullet$       &    $1$        &  $uv$ &  $u^2 v^2$ & $u^3 v^3$ &  $u$ &  $v $ & $u^2 v$ &   $uv^2 $   \\ \hline\hline
$1$              &     $1$        &  $uv$ &  $u^2 v^2$ & $u^3 v^3$ &  $u$ &  $v $ & $u^2 v$ &   $uv^2 $ \\ \hline
$uv$          &   $uv$ & $u^2 v^2$ &  $u^3 v^3$ & $0$ & $u^2v$ & $uv^2$ & $0$ &     $0$       \\ \hline

$u^2 v^2$      &   $u^2 v^2$    & $u^3 v^3$      & $0$          & $0$ &  $0$        &   $0$     &     $0$         &      $0$           \\ \hline
$u^3 v^3$      &   $u^3 v^3$        & $0$      & $0$          & $0$      & $0$ &     $0$         &     $0$             &           $0$               \\ \hline

$u$          &   $u$    & $u^2 v$      &      $0$      & $0$      & $-u^2 v^2 $     &     $0$         &        $-u^3v ^3$             &    $0$                     \\ \hline

$v$          &   $v$    & $uv^2$      & $0$        & $0$      & $0$     &    $-u^2 v^2$           &        $0$          &       $- u^3v ^3 $           \\ \hline

$u^2 v$   &   $u^2 v$     & $0$     & $0$          & $0$      & $-u^3v ^3 $      &    $0$          &        $0$          &           $0$              \\ \hline

$uv^2$   &   $uv^2$     & $0$       & $0$          & $0$      & $0$     &     $-u^3v ^3 $         &       $0$           &             $0$             \\ \hline

\end{tabular}

\end{center}

\noindent On the orbifold cohomology side, we have the sector decomposition $I_{\mathcal{X}}=\coprod_{v\in\Box (P)}\mathcal{X}_v =:\coprod_{i=0,\cdots 5}\mathcal{X}_{v_i}$ of the inertia orbifold of $\mathcal{X}$. A basis of the orbifold cohomology of $\mathcal{X}$ is given by 
$$1_{v_0}, 1_{v_0}p,1_{v_1},1_{v_2}, 1_{v_3}, 1_{v_3}p, 1_{v_4},1_{v_5}$$
 with 
respective grading $0, 1, \frac{1}{2}, 1, \frac{1}{2}, \frac{3}{2} , 1, \frac{1}{2}$.
Let us define
the isomorphism
$$H_{\orb}^{2*}(\mathcal{X}, \qit )\longrightarrow \gr (\frac{\mathcal{B}}{\mathcal{L}})$$
by the assignments 
$$1_{v_0} \mapsto 1, 1_{v_1} \mapsto u , 1_{v_2} \mapsto u^2 v, 1_{v_3} \mapsto uv, 1_{v_4} \mapsto uv^2 , 1_{v_5} \mapsto v, 1_{v_0} p \mapsto u^2 v^2 , 1_{v_3} p \mapsto u^3 v^3.$$
We would get the following table for an "orbifold" cup-product in $H_{\orb}^{2*}(\mathcal{X}, \qit )$:

\begin{center}

\begin{tabular}{|c|c|c|c|c|c|c|c|c|} \hline\hline
$ \cup_{\orb}$       &    $1_{v_0}$  &  $1_{v_1} $ &  $1_{v_2}$ & $1_{v_3}$ &  $1_{v_4}$ &  $1_{v_5}$ & $1_{v_0}p$ &   $1_{v_3}p$   \\ \hline\hline
$1_{v_0}$              &     $1_{v_0}$  &  $1_{v_1}$  & $1_{v_2}$ & $1_{v_3}$ & $1_{v_4}$  &   $1_{v_5}$  &  $1_{v_0}p$ &   $1_{v_3}p$   \\ \hline
$1_{v_1}$              &     $1_{v_1}$  &  $-1_{v_0}p$  & $-1_{v_3}p$ & $1_{v_2}$ & $0$  &   $0$  &  $0$ &   $0$   \\ \hline
$1_{v_2}$              &     $1_{v_2}$  &  $-1_{v_3}p$  & $0$ & $0$ & $0$  &   $0$  &  $0$ &   $0$   \\ \hline
$1_{v_3}$              &     $1_{v_3}$  &  $1_{v_2}$  & $0$ & $1_{v_0}p$ & $0$  &   $1_{v_4}$  &  $1_{v_3}p$ &   $0$   \\ \hline
$1_{v_4}$              &     $1_{v_4}$  &  $0$  & $0$ & $0$ & $0$  &   $1_{v_3}p$  &  $0$ &   $0$   \\ \hline
$1_{v_5}$              &     $1_{v_5}$  &  $0$  & $0$ & $1_{v_4}$ & $1_{v_3}p$  &   $-1_{v_0}p$  &  $0$ &   $0$   \\ \hline
$1_{v_0}p$                       &     $1_{v_0}p$  &  $0$  & $0$ & $1_{v_3}p$ & $0$  &   $0$  &  $0$ &   $0$   \\ \hline
$1_{v_3}p$            &     $1_{v_3}p$  &  $0$  & $0$ & $0$ & $0$  &   $0$  &  $0$ &   $0$   \\ \hline
\end{tabular}

\end{center}

\end{example}

\subsection{Toric Newton spectrum and Ehrhart theory}

\label{sec:Combinatoire}

 By Theorem \ref{theo:ToricSpectrumSeries}, the results of \cite[Section 4]{D12} still hold, 
with the same proofs, for the toric Newton spectrum of the Newton polytope of a convenient and nondegenerate polynomials. In words,
the toric Newton spectrum of a polynomial 
counts weighted lattice points in its Newton polytope $P$: define, for a nonnegative integer $\ell$, $L_P (\ell ):= \card ( (\ell P )\cap M)$. Then $L_P$ is a polynomial in $\ell$ of degree $n$ (the {\em Ehrhart polynomial} of $P$) and we have
\begin{equation}\label{eq:serie Ehrhart}
1+\sum_{m\geq 1}L_P (m) z^m =\frac{\delta_0 +\delta_1 z +\cdots +\delta_n z^n}{(1-z)^{n+1}}
\end{equation}
where the $\delta_j$'s are nonnegative integers. 
We will write $\delta_P (z) :=\delta_0 +\delta_1 z +\cdots +\delta_n z^n$: this is the $\delta$-vector of $P$.

\begin{theorem}\label{coro:SpecEgalDelta}\cite[Section 4.3]{D12}
Let $f$ be a convenient and nondegenerate polynomial on $\cit^n$ and let $P$ be its Newton polytope.
Let $\delta_P (z) :=\sum_{k=0}^n \delta_k z^k $ be the $\delta$-vector of the $P$ and let 
$P_{\gr^{\mathcal{N}}(\frac{\mathcal{B}}{\mathcal{L}})}(z)=\sum_{i=1}^{\mu_P} z^{\beta_i}$ be the toric Newton spectrum of $f$. Then, for $k=0,\cdots ,n$, the coefficient $\delta_k$ is equal to the number of $\beta_i$'s such that $\beta_i\in ]k-1,k]$.
\end{theorem}

\begin{example} Theorem \ref{coro:SpecEgalDelta} gives a recipe in order to calculate Ehrhart polynomials of Newton polytopes of polynomials. In order to emphasize this, we give two very simple examples:
\begin{enumerate}
\item Let $f(u_1 , \cdots ,u_n )=u_1 +\cdots  +u_n$ on $\cit^n$. 
Its Newton polytope is
$$P :=\conv ((0,\cdots ,0), (1,0,\cdots ,0), (0,1,0,\cdots ,0), (0,\cdots ,0,1)),$$
the standard simplex in $\rit^n$. 
The toric Newton spectrum of $f$ is $z^0$ and it follows that the $\delta$-vector 
of $P$ is given by $\delta_0 =1$, $\delta_2 =\cdots =\delta_n =0$: its Ehrhart polynomial is 
$L_P (z) =\binom{z+n}{n}$.
\item Let $f(u_1 , u_2 , u_3 )=u_1 +u_2  +u_3^c$ where $c$ is a positive integer.  
Its Newton polytope is $P :=\conv ((0,0,0), (1,0,0), (0,1,0), (0,0,c))$. 
The toric Newton spectrum of $f$ is $\sum_{i=0}^{c-1} z^{i/c}$. Therefore, the $\delta$-vector 
of $P$ is given by $\delta_0 =1$, $\delta_1 =c-1$ and $\delta_2 =\delta_3 =0$: its Ehrhart polynomial is 
$L_P (z) =\binom{z+3}{3}+(c-1)\binom{z+2}{3}$.
\end{enumerate}
\end{example}

\subsection{The local case}

\label{sec:LocalCase}

Analogous results hold true in the local case:
let now $f$ be a power series with an isolated critical point at the origin, and let $\mu_0 :=\dim_{\cit} \cit \{u_1 ,\cdots ,u_n \}/(\partial_{u_1}f ,\cdots , \partial_{u_n}f)$ be its Milnor number at 
the origin.
Let $\Gamma_+ (f)$ be the convex hull of 
$\cup_{a\in\ supp(f)-0}(a+\rit_+^{n})$ in $\rit_+^n$, let $\Gamma (f)$ be the union of the compact faces of $\Gamma_+ (f)$ and let $P$ be the union of all the segments starting from the origin and ending on $\Gamma (f)$: this is the {\em Newton polyhedron} of $f$. We will put $\mu_P :=n! \vol (P)$ where
the volume $\vol (P)$ of the Newton polyhedron $P$ is normalized such that the volume of the cube is equal to one.

A Newton  function $\nu_0 $ is defined as in Section \ref{sec:NewtonForms} putting $\nu_0 (a) =\min_{F} <u_F ,a>$ for $a\in\nit^n$, see \cite[Section 2.1]{K}. It now satisfies 
$\nu_0 (a+b) \geq \nu_0 (a) +\nu_0 (b)$ for $a,b\in\nit^n$, with equality if $a$ and $b$ belong to the same cone. This provides a decreasing Newton filtration $\mathcal{N}_0$ on $\mathcal{A}:=\cit \{u_1 ,\cdots ,u_n \}$, a local toric Newton spectrum
$P_{\gr_{\mathcal{N}_0}(\frac{\mathcal{A}}{\mathcal{L}})}(z)$, 
where $\mathcal{L}$ still denotes the ideal generated by the partial derivative 
$u_1 \frac{\partial f}{\partial u_1 }, \cdots , u_n \frac{\partial f}{\partial u_n }$ of $f$,
and a local Newton spectrum $\Spec_f^0 (z)$, which is equal to the calssical singularity spectrum. 
The previous results still hold with minor modifications.
In particular, let $\Sigma$ be the fan built over the faces of the Newton boundary of $P$ (the union of the closed faces of $P$ that do not contain the origin), and let $X_{\Sigma}$ be the toric variety associated with the fan $\Sigma$. 
Assume that $X_{\Sigma}$ is simplicial and let $\mathcal{X}$ be the Deligne-Mumford stack
associated by \cite{BCS} with the stacky fan $\mathbf{\Sigma}=(\zit^n, \Sigma , \mathcal{V} (P))$ 
where the $\mathcal{V} (P)$ is the set of the vertices of $P$ different from the origin.

\begin{theorem}
Let $f\in\mathcal{A}$ be a convenient and nondegenerate power series. 
 Assume that the fan $\Sigma$ of its Newton polyhedron is simplicial. Then $P_{\gr_{\mathcal{N}_0}(\frac{\mathcal{A}}{\mathcal{L}})}(z)=\sum_{\alpha\in\qit} \dim_{\qit} H^{2\alpha }_{\orb} (\mathcal{X}, \qit ) z^{\alpha}$.\qed
\end{theorem}

Although $P$ is not a polytope in general, we can count lattice points in it. Ehrhart polynomials and $\delta$-vectors are still defined (see \cite{BMc} for instance) and we have:

 \begin{theorem}\label{coro:SpecEgalDeltaLocal}
Let $f$ be a convenient and nondegenerate power series on $\cit^n$ and let $P$ be its Newton polyhedron (assumed to be simplicial).
Let $\delta_P (z) :=\sum_{k=0}^n \delta_k z^k $ be the $\delta$-vector of the $P$ and let 
$P_{\gr_{\mathcal{N}_0}(\frac{\mathcal{A}}{\mathcal{L}})}(z)=\sum_{i=1}^{\mu_P} z^{\beta_i}$ be the local toric Newton spectrum of $f$. Then, for $k=0,\cdots ,n$, the coefficient $\delta_k$ is equal to the number of $\beta_i$'s such that $\beta_i\in ]k-1,k]$.
\end{theorem}

\begin{example} 
Let us consider the classical example $f(x,y)= x^5 + x^2 y^2 + y^5$ on $\cit^2$. Its Milnor number at the origin is $\mu_0 =11$ and its 
toric Milnor number is $\mu_P =20$. Its local toric Newton spectrum (at the origin) 
is
$$P_{\gr_{\mathcal{N}_0}(\frac{\mathcal{A}}{\mathcal{L}})}(z)=1+2z^{1/5}+2z^{2/5} +2z^{3/5}+ 2z^{4/5}+
z^{1/2}+z+z^{3/2} +2z^{7/10} +2 z^{9/10} +2 z^{11/10} +2z^{13/10}.$$  
 Therefore, the $\delta$-vector 
of $P$ is given by $\delta_P (z) =1+14z +5 z^2$ and its Ehrhart polynomial is 
$L_P (z) =\binom{z+2}{2}+14\binom{z+1}{2}+5\binom{z}{2}$. Notice that its local Newton spectrum is
$$\Spec_f^0 (z)=z^{1/2}+z+z^{3/2} +2z^{7/10} +2 z^{9/10} +2 z^{11/10} +2z^{13/10}$$
and this formula can be deduced from the description of $P_{\gr_{\mathcal{N}_0}(\frac{\mathcal{A}}{\mathcal{L}})}(z)$.
\end{example}


\begin{thebibliography}{999} 
\bibitem[1]{ALR} Adem, A., Leida, J., Ruan, Y.: Orbifolds and stringy topology, Cambridge Tracts in Math., {\bf 171}, 2007.
\bibitem[2]{BMc} U. Betke, U., McMullen, P.:  Lattice points in lattice polytopes, Monatsh. Math., {\bf 99} (4), 1985, p. 253–265.
\bibitem[3]{BCS} Borisov, L., Chen, L.,  Smith, G.: The orbifold Chow ring of toric Deligne-Mumford stacks, 
J. Amer. Soc. , {\bf 18} (1), 2005, p. 193-215. 
\bibitem[4]{BGMM} Brian\c{c}on, J., Granger, M., Maisonobe, Ph., Miniconi, M.: Algorithme de calcul du polyn\^ome de Bernstein : cas non-d\'eg\'en\'er\'e,  Ann. Instit. Fourier , {\bf 39} (3), 1989, p. 553-610.
\bibitem[5]{CLS} Cox, D., Little, J., Schenck, A.: Toric varieties, American Mathematical Society, {\bf 124}, 2010.
\bibitem[6]{Da} Danilov, V.I.: The geometry of toric varieties, Russian Math. Survey, {\bf 33}, 1978, p. 97-154.
\bibitem[7]{D12} Douai, A.: Ehrhart polynomials of Newton polytopes and spectrum at infinity of Laurent polynomials, arXiv:1811.07724.
\bibitem[8]{D0} Douai, A.: Global spectra, polytopes and stacky invariants, Math. Z.,  {\bf 288} (3), 2018,  p. 889-913. 
\bibitem[9]{DoSa1} Douai, A., Sabbah, C.: Gauss-Manin systems, Brieskorn
lattices and Frobenius structures I. Ann. Inst. Fourier, {\bf 53} (4), 2003, p. 1055-1116.
\bibitem[10]{KV} Khovanskii, A.G., Varchenko, A.N.: Asymptotics of integrals over vanishing cycles and the Newton polyhedron, 
Sov. Math. Dokl., {\bf 32}, 1985, p. 122-127.
\bibitem[11]{K} Kouchnirenko, A.G.: Poly\`edres de Newton et nombres de Milnor, Invent. Math.  {\bf 32}, 1976, p. 1-31.
\bibitem[12]{MSTakeu} Matsui, Y., Takeuchi, K.: Monodromy at infinity of polynomial maps and Newton polyhedra (with Appendix by C. Sabbah). IMRN, {\bf 2013-8},  2013, p. 1691-1746.
\bibitem[13]{Sab} Sabbah, C.: Hypergeometric periods for a tame polynomial. Portugalia Mathematicae, {\bf 63}, 2006, p. 173-226.
\bibitem[14]{MS} Saito, M.: Exponents and Newton polyhedra of isolated hypersurface singularities. Math. Ann., {\bf 281}, 1988, p. 411-417.      
\bibitem[15]{Stapledon} Stapledon, A.: Weighted Ehrhart Theory and Orbifold Cohomology, Adv. Math., {\bf 219},  2008, p. 63-88. 
\end{thebibliography}
\end{document}